\pdfoutput=1

\documentclass[11pt]{amsart}

\usepackage{amsmath,amsthm,amssymb}
\usepackage{subfig}
\usepackage{graphicx}
\usepackage{doi}
\usepackage[margin=1.1in]{geometry}
\usepackage{hhline}

\begin{document}

\theoremstyle{plain}
\newtheorem{Theorem}{Theorem}
\newtheorem{Lemma}[Theorem]{Lemma}
\newtheorem{Corollary}[Theorem]{Corollary}
\newtheorem{Proposition}[Theorem]{Proposition}

\renewcommand\arraystretch{1.4}

\newcommand{\R}{\mathbb{R}}  
\newcommand{\Z}{\mathbb{Z}}  
\newcommand{\N}{\mathbb{N}}

%=================================================================

\title{Some characteristics of the simple Boolean quadric polytope extension}

\author{Andrei Nikolaev}
\thanks {The research was partially supported by the Russian Foundation for Basic Research, Project 14-01-00333, the President of Russian Federation Grant MK-5400.2015.1, and the initiative R\&D VIP-004 YSU}

\address{%
Department of Discrete Analysis, P.G. Demidov Yaroslavl State University, Sovetskaya, 14, Yaroslavl, 150000, Russia
}
\email {andrei.v.nikolaev@gmail.com}

\begin{abstract}
Following the seminal work of Padberg on the Boolean quadric polytope $BQP$ and its LP relaxation $BQP_{LP}$, we consider a natural extension: $SATP$ and $SATP_{LP}$ polytopes, with $BQP_{LP}$ being projection of the $SATP_{LP}$ face (and $BQP$ -- projection of the $SATP$ face). 
We consider a problem of integer recognition: determine whether a maximum of a linear objective function is achieved at an integral vertex of a polytope. 
Various special instances of 3-SAT problem like NAE-3-SAT, 1-in-3-SAT, weighted MAX-3-SAT, and others can be solved by integer recognition over $SATP_{LP}$.
We describe all integral vertices of $SATP_{LP}$. Like $BQP_{LP}$, polytope $SATP_{LP}$ has the Trubin-property being quasi-integral (1-skeleton of $SATP$ is a subset of 1-skeleton of $SATP_{LP}$). However, unlike $BQP$, not all vertices of $SATP$ are pairwise adjacent, the diameter of $SATP$ equals 2, and the clique number of 1-skeleton is superpolynomial in dimension.
It is known that the fractional vertices of $BQP_{LP}$ are half-integer (0, 1 or 1/2 valued). We show that the denominators of $SATP_{LP}$ fractional vertices can take any integral value. 
Finally, we describe polynomially solvable subproblems of integer recognition over $SATP_{LP}$ with constrained objective functions.
Based on that, we solve some cases of edge constrained bipartite graph coloring.
\end {abstract}


\maketitle

%%%%%%%%%%%%%%%%%%%%%%%%%%%%%%%%%%%%%%%%%%

\section{Boolean quadric polytope and its relaxations}

We consider the well-known Boolean quadric polytope $BQP(n)$ \cite {Padberg}, satisfying the constraints
\begin{gather}
x_{i} + x_{j} - x_{i,j} \leq 1, \label {BQPfirst}\\
x_{i,j} \leq x_{i},\\
x_{i,j} \leq x_{j},\\
x_{i,j} \geq 0, \label{BQPlast}\\
x_{i}, x_{i,j} \in \{0,1\}, \label {BQPintegral}
\end {gather}
for all $i,j: \ 1\leq i < j \leq n$.

Polytope $BQP(n)$ is constructed from the NP-hard problem of unconstrained Boolean quadratic programming:
$$Q(x) = x^{T} Q x \rightarrow \max,$$
where vector $x\in \{0,1\}^{n}$, and $Q$ is an upper triangular matrix, by introducing new variables $x_{i,j} = x_{i} x_{j}$.

Boolean quadric polytope arises in many fields of mathematics and physics. Sometimes it is called the correlation polytope, since its members can be interpreted as joint correlations of events in some probability space. Also within the quantum mechanics Boolean quadric polytope is connected with the representability problem for density matrices of order $2$ that render physical properties of a system of particles \cite {Deza-Laurent}. 
Besides, $BQP(n)$ is in one-to-one correspondence via the covariance linear mapping with the well-known cut polytope $CUT(n+1)$ of the complete graph on $n+1$ vertices \cite {De Simone} (see also \cite {Barahona-Mahjoub}).

In recent years, the Boolean quadric polytope has been under the close attention in connection with the problem of estimating the extension complexity \cite {Maksimenko}. An extension of the polytope $P$ is another polytope $Q$ such that $P$ is the image of $Q$ under a linear map. 
The number of facets of $Q$ is called the size of an extension. Extension complexity of $P$ is defined as the minimum size of all possible extensions. Fiorini et al. proved that the extension complexity of the Boolean quadric polytope is exponential \cite {Fiorini} (see also \cite {Kaibel}).

\begin {Theorem}
The extension complexity of $BQP(n)$ and $CUT(n)$ is $2^{\Omega (n)}$.
\end {Theorem}

Since polytopes of many combinatorial optimization problems, including stable set, knapsack, $3$-dimensional matching, and traveling salesman, contain a face that is an extension of $BQP(n)$, those polytopes also have an exponential extension complexity. Thus, corresponding problems can not be solved effectively by linear programming, as any LP formulation will have an exponential number of inequalities.

If we exclude from the system (\ref{BQPfirst})-(\ref{BQPintegral}) the constraints (\ref{BQPintegral}) that the variables are integral, the remaining system (\ref{BQPfirst})-(\ref{BQPlast}) describes the Boolean quadric relaxation polytope $BQP_{LP}(n)$. Corresponding cut polytope relaxation is known as the rooted semimetric polytope $RMET(n)$. 

When we add the slack variables
\begin{gather*}
x^{1,1}_{i,j} = x_{i,j},\ \ \ \ x^{2,2}_{i,i} = 1 - x_{i,i},\\
x^{1,2}_{i,j} = x_{j,j} - x_{i,j},\ \ \ x^{2,1}_{i,j} = x_{i,i} - x_{i,j},\\
x^{2,2}_{i,j} = 1 - x_{i,i} - x_{j,j} + x_{i,j},
\end {gather*}
$BQP_{LP}(n)$ can be written in the standard form
\begin{gather}
x^{1,1}_{i,j}+x^{1,2}_{i,j}+x^{2,1}_{i,j}+x^{2,2}_{i,j} = 1, \label {BQP_standard_first} \\
x^{1,1}_{i,j}+x^{1,2}_{i,j} = x^{1,1}_{k,j}+x^{1,2}_{k,j}, \\
x^{1,1}_{i,j}+x^{2,1}_{i,j} = x^{1,1}_{i,l}+x^{2,1}_{i,l}, \\
x^{1,2}_{i,i}=x^{2,1}_{i,i}=0, \\
x^{1,1}_{i,j}\geq 0, \ x^{1,2}_{i,j}\geq 0, \ x^{2,1}_{i,j}\geq 0, \ x^{2,2}_{i,j}\geq 0, \label {BQP_standard_last}
\end{gather}
where $1 \leq k \leq i \leq j \leq l \leq n$ \cite {Bondarenko-Uryvaev}.

Points of the $BQP_{LP}(n)$ polytope in the form (\ref{BQP_standard_first})-(\ref{BQP_standard_last}) can be conveniently represented as a block upper triangular matrix (Table \ref{BQP_LP_block}).
\begin{table}[h]
\centering
\begin {tabular} {||c|c||c|c||}
\hhline {|t:==:t:==t:|}
$x^{1,1}_{i,i}$ & $0$ & $x^{1,1}_{i,j}$ & $x^{1,2}_{i,j}$ \\ \hhline {||--||--||}
$0$ & $x^{2,2}_{i,i}$ & $x^{2,1}_{i,j}$ & $x^{2,2}_{i,j}$ \\ \hhline {|b:==::==:|}
\multicolumn{2}{c||}{} & $x^{1,1}_{j,j}$ & $0$  \\ \hhline {~~||--||}
\multicolumn{2}{c||}{} & $0$ & $x^{2,2}_{j,j}$  \\ \hhline {~~|b:==:b|}
\end {tabular}
\caption{Fragment of the $BQP_{LP}(n)$ block matrix.}
\label {BQP_LP_block}
\end {table}

The relaxation polytope $BQP_{LP}(n)$ and the Boolean quadric polytope $BQP(n)$ have the same integral vertices. Hence, the Boolean quadratic programming and max-cut are reduced to integer programming over $BQP_{LP}(n)$.

\begin {Theorem}
Integer programming over $BQP_{LP}(n)$ is NP-hard.
\end {Theorem}

We consider a problem of integer recognition: for a given linear objective function $f(x)$ and a polytope $P$ determine whether $\max \{f(x)\ | \ x \in P\}$ is achieved at an integral vertex of $P$. It is similar to the integer feasibility problem and NP-complete in general case. In \cite {Bondarenko-Uryvaev} integer recognition over $BQP_{LP}(n)$ was solved by linear programming over $BQP_{LP}(n)$ and the metric polytope $MET(n)$, obtained by augmenting the system (\ref{BQP_standard_first})-(\ref{BQP_standard_last}) by the triangle inequalities that define the facets of $BQP(3)$ \cite {Padberg}:
\begin{gather*}
x_{i,i} + x_{j,j} + x_{k,k} - x_{i,j} - x_{i,k} - x_{j,k} \leq 1,\\
-x_{i,i} + x_{i,j} + x_{i,k} - x_{j,k} \leq 0,\\
-x_{j,j} + x_{i,j} - x_{i,k} + x_{j,k} \leq 0,\\
-x_{k,k} - x_{i,j} + x_{i,k} + x_{j,k} \leq 0,
\end {gather*}
for all $i,j,k$, where $1 \leq i < j < k \leq n$.

\begin {Lemma}
(see \cite {Bondarenko-Uryvaev}) If for some linear objective function $f(x)$ we have 
$$\max_{x\in BQP_{LP}(n)} f(x) = \max_{x\in MET(n)} f(x),$$
then the maximum is achieved at an integral vertex of $BQP_{LP}(n)$. Otherwise,
$$\max_{x\in BQP_{LP}(n)} f(x) > \max_{x\in MET(n)} f(x),$$
and the function $f(x)$ has a maximum value at the face containing only fractional vertices.
\label {lemma_cut_BQPLP}
\end {Lemma}

Hence, we have

\begin {Theorem}
Integer recognition over $BQP_{LP}(n)$ is polynomially solvable.
\label {theorem_BQPLP_polynomial}
\end {Theorem}

Metric polytope $MET(n)$ itself is also important, since it is the most simple and natural relaxation of the $CUT(n)$ polytope, and has many practical applications, such as being a compact LP formulation for the max-cut problem on graphs not contractible to $K_{5}$ \cite {Barahona}. Integer recognition over metric polytope is examined in \cite {Bondarenko-NSS}.

Note that integer programming and integer recognition problems over polytope $BQP_{LP}(n)$ differ greatly in their complexity.

For any polytope $P$, we call the collection of its vertices (0-faces) and its edges (1-faces) the 1-skeleton of $P$. 
Let $Q$ be a polytope that is contained in $P$. We say that $P$ has the Trubin-property (with respect to $Q$) if the 1-skeleton of $Q$ is a subset of the 1-skeleton of $P$ \cite {Padberg}. Polytope $P$ with this property is also called quasi-integral. If $P$ has the Trubin property, then all vertices of $Q$ are vertices of $P$ and those facets of $Q$ that define invalid inequalities for $P$ do not create any new adjacencies among the vertices of $Q$.

\begin {Theorem} (see \cite {Padberg})
The diameter of $BQP(n)$ equals $1$. Both relaxations $BQP_{LP}(n)$ and $MET(n)$ have the Trubin-property with respect to $BQP(n)$.
\end {Theorem}

As for fractional vertices the properties of $BQP_{LP}(n)$ and $MET(n)$ are completely different.

\begin {Theorem} (see \cite {Padberg})
Every vertex of $BQP_{LP}(n)$ is $\{0,\frac{1}{2},1\}$ valued. 
\end {Theorem}

\begin {Theorem} (see \cite {Laurent})
There are vertices of $MET(n)$ that take values $\frac{a}{n+1}$ for $a = 1,2,\ldots,n$.
\label {theorem_MET_vertices}
\end {Theorem}

Thus, the denominators of the $MET(n)$ vertices can take any integral values, unlike the vertices of $BQP_{LP}(n)$.

The results of this paper were first presented at the 9th International Conference ``Discrete Optimization and Operations Research'', Vladivostok, Russia, 2016. \cite {Nikolaev-DOOR}.

%%%%%%%%%%%%%%%%%%%%%%%%%%%%%%%%%%%%%%%%%%

\section {3-SAT relaxation polytope}
\label {SAT_polytope_section}

We consider a more general polytope $SATP(m,n) \subset \R^{6mn}$ (see \cite {Bondarenko-Uryvaev}), obtained as the convex hull of all integral solutions of the system
\begin{align}
\sum_{k,l}x^{k,l}_{i,j}&=1, \label {SATeqfirst} \\
	x^{1,1}_{i,j}+x^{2,1}_{i,j}+x^{3,1}_{i,j}&=x^{1,1}_{i,t}+x^{2,1}_{i,t}+x^{3,1}_{i,t}, \label {SATeqsecond}\\
	x^{k,1}_{i,j}+x^{k,2}_{i,j}&=x^{k,1}_{s,j}+x^{k,2}_{s,j}, \label {SATeqthird} \\
	x^{k,l}_{i,j}&\geq 0, \label {SATineq}
\end{align}
where $k=1,2,3$; $l=1,2$; $i,s=1,\ldots m$; $j,t=1,\ldots n$. 

Inequalities (\ref {SATeqfirst})-(\ref {SATineq}) without the integrality constraint define LP relaxation $SATP_{LP}(m,n)$.
Points that satisfy the system can be conveniently represented as a block matrix (Table \ref{SATPmatrix}).
\begin{table}[h]
\centering
\begin {tabular} {||c|c||c|c||}
\hhline {|t:==:t:==t:|}
$x^{1,1}_{i,j}$ & $x^{1,2}_{i,j}$ & $x^{1,1}_{i,t}$ & $x^{1,2}_{i,t}$ \\ \hhline {||--||--||}
$x^{2,1}_{i,j}$ & $x^{2,2}_{i,j}$ & $x^{2,1}_{i,t}$ & $x^{2,2}_{i,t}$ \\ \hhline {||--||--||}
$x^{3,1}_{i,j}$ & $x^{3,2}_{i,j}$ & $x^{3,1}_{i,t}$ & $x^{3,2}_{i,t}$ \\ \hhline {|:==::==:|}

$x^{1,1}_{s,j}$ & $x^{1,2}_{s,j}$ & $x^{1,1}_{s,t}$ & $x^{1,2}_{s,t}$ \\ \hhline {||--||--||}
$x^{2,1}_{s,j}$ & $x^{2,2}_{s,j}$ & $x^{2,1}_{s,t}$ & $x^{2,2}_{s,t}$ \\ \hhline {||--||--||}
$x^{3,1}_{s,j}$ & $x^{3,2}_{s,j}$ & $x^{3,1}_{s,t}$ & $x^{3,2}_{s,t}$ \\ \hhline {|b:==:b:==:b|}
\end {tabular}
\caption{Fragment of the $SATP_{LP}(m,n)$ block matrix}
\label {SATPmatrix}
\end {table}

If we consider a face of the $SATP(n,n)$ polytope, constructed as follows:
\begin{gather*}
\forall i,j: x^{3,1}_{i,j} = x^{3,2}_{i,j} = 0,\\
\forall i: x^{1,2}_{i,i} = x^{2,1}_{i,i} = 0,
\end {gather*}
and discard all the coordinates for $i<j$ (orthogonal projection), then we get the polytope $BQP(n)$.
As a result, we have

\begin {Theorem}
The extension complexity of the polytope $SATP(m,n)$ is $2^{\Omega (\min \{m,n\})}$.
\end {Theorem}

In \cite {Bondarenko-Uryvaev} by reduction from 3-SAT it was shown that 
\begin {Theorem}
Integer recognition over $SATP_{LP}(m,n)$ is NP-complete.
\label {Integer_recognition_SATP_NPC}
\end {Theorem}

We prove that the polytope $SATP_{LP}(m,n)$ can be seen as a LP relaxation of various special instances of 3-SAT problem as well.

\begin {Lemma}
\label {integral_vertices}
Let $z$ be the vertex of the $SATP(m,n)$ polytope, then its coordinates are determined by the vectors $\textbf{row} (z) \in \{0,1\}^{m}$ and $\textbf{col} (z) \in \{0,1,2\}^{n}$ by the following formulas:
\begin{align}
  x^{1,1}_{i,j} &= \frac {1}{2} (1-\textbf{row}_{i}(z)) (2-\textbf{col}_{j}(z)) (1-\textbf{col} _{j}(z)), \label {x11integral} \\
	x^{1,2}_{i,j} &= \frac {1}{2} \textbf{row}_{i}(z) (2-\textbf{col} _{j}(z)) (1-\textbf{col} _{j}(z)), \\
	x^{2,1}_{i,j} &= (1-\textbf{row}_{i}(z)) \textbf{col}_{j}(z) (2-\textbf{col} _{j}(z)), \\
	x^{2,2}_{i,j} &= \textbf{row}_{i}(z) \textbf{col}_{j}(z) (2-\textbf{col}_{j}(z)), \\
	x^{3,1}_{i,j} &= \frac {1}{2} (1-\textbf{row}_{i}(z)) \textbf{col}_{j}(z) (1-\textbf{col}_{j}(z)), \\
	x^{3,2}_{i,j} &= \frac {1}{2} \textbf{row}_{i}(z) \textbf{col}_{j}(z) (1-\textbf{col}_{j}(z)). \label {x32integral}
\end{align}
\end {Lemma}

\begin {proof}
From the constraints (\ref {SATeqfirst})-(\ref {SATineq}) it follows that the vertices of the polytope $SATP(m,n)$ are zero-one points with exactly one unit per block. 
For any vertex $z$ of $SATP(m,n)$ we define $\textbf {\emph {row}} (z) \in \{0,1\}^{m}$ and $\textbf {\emph {col}} (z)\in \{0,1,2\}^{n}$ vectors by the following rules:
\begin {gather*}
\textbf {\emph {row}}_{i} (z) = \left[
\begin {array} {l}
0,\ \mbox{if} \ x^{1,1}_{i,1} + x^{2,1}_{i,1} + x^{3,1}_{i,1} = 1, \\
1,\ \mbox {otherwise.}
\end {array}
\right.
\\
\textbf {\emph {col}}_{j} (z) = \left[
\begin {array} {l}
0,\ \mbox{if} \ x^{1,1}_{1,j} + x^{1,2}_{1,j} = 1, \\
1,\ \mbox{if} \ x^{2,1}_{1,j} + x^{2,2}_{1,j} = 1, \\
2,\ \mbox {otherwise.}
\end {array}
\right.
\end {gather*}
 
All the vertex coordinates are uniquely determined by the first row and first column of blocks from the system (\ref {SATeqfirst})-(\ref {SATeqthird}). Equations (\ref{x11integral})-(\ref{x32integral}) correspond to them for zero-one points. Thus, the polytope $SATP(m,n)$ has exactly $2^{m}3^{n}$ vertices.
\end {proof}

We consider a classical MAX-3SAT problem: given a set $U=\{u_{1}, \ldots, u_{m}\}$ of variables and a collection $C=\{c_{1}, \ldots, c_{n}\}$ of 3-literal clauses over $U$, find a truth assignment that satisfies the largest number of clauses.

With each instance of the problem we associate an objective vector $v\in \R^{6mn}$:
\begin {itemize}
\item if a clause $c_{j}$ has a literal $u_{i}$ at the place $k$, then $v^{k,1}_{i,j} = 1$,
\item if a clause $c_{j}$ has a literal $\overline{u}_{i}$ at the place $k$, then $v^{k,2}_{i,j} = 1$,
\item all the remaining coordinates of the vector $v$ equal to $0$.
\end {itemize} 

An example of an objective vector $v$ for a formula
\begin {equation}
(x \vee y \vee \overline{z})\wedge (\overline{x} \vee z \vee t) \wedge (\overline{y} \vee z \vee \overline{t})
\label {boolean_formula}
\end {equation}
is shown in Table \ref{SATreduction} (a).
\begingroup
\renewcommand*{\arraystretch}{0.9}
\begin{table}[h]
\centering
\subfloat[MAX3SAT]{
  \centering
	\begin{tabular}{||c|c||c|c||c|c||}
	\hhline {|t:==:t:==:t:==:t|}
	 $1$ & $0$ & $0$ & $1$ & $0$ & $0$ \\ \hhline {||--||--||--||}
	 $0$ & $0$ & $0$ & $0$ & $0$ & $0$ \\ \hhline {||--||--||--||}
	 $0$ & $0$ & $0$ & $0$ & $0$ & $0$ \\
	\hhline {|:==::==::==:|}
	 $0$ & $0$ & $0$ & $0$ & $0$ & $1$ \\ \hhline {||--||--||--||}
	 $1$ & $0$ & $0$ & $0$ & $0$ & $0$ \\ \hhline {||--||--||--||}
	 $0$ & $0$ & $0$ & $0$ & $0$ & $0$ \\	 
	\hhline {|:==::==::==:|}
	 $0$ & $0$ & $0$ & $0$ & $0$ & $0$ \\ \hhline {||--||--||--||}
	 $0$ & $0$ & $1$ & $0$ & $1$ & $0$ \\ \hhline {||--||--||--||}
	 $0$ & $1$ & $0$ & $0$ & $0$ & $0$ \\	
	\hhline {|:==::==::==:|}
	 $0$ & $0$ & $0$ & $0$ & $0$ & $0$ \\ \hhline {||--||--||--||}
	 $0$ & $0$ & $0$ & $0$ & $0$ & $0$ \\ \hhline {||--||--||--||}
	 $0$ & $0$ & $1$ & $0$ & $0$ & $1$ \\	
	\hhline {|b:==:b:==:b:==:b|}
	\end{tabular}
}
\ \ \
\subfloat[X3SAT]{
\centering
	\begin{tabular}{||c|c||c|c||c|c||}
	\hhline {|t:==:t:==:t:==:t|}
	 $1$ & $0$ & $0$ & $1$ & $0$ & $0$ \\ \hhline {||--||--||--||}
	 $0$ & $1$ & $1$ & $0$ & $0$ & $0$ \\ \hhline {||--||--||--||}
	 $0$ & $1$ & $1$ & $0$ & $0$ & $0$ \\
	\hhline {|:==::==::==:|}
	 $0$ & $1$ & $0$ & $0$ & $0$ & $1$ \\ \hhline {||--||--||--||}
	 $1$ & $0$ & $0$ & $0$ & $1$ & $0$ \\ \hhline {||--||--||--||}
	 $0$ & $1$ & $0$ & $0$ & $1$ & $0$ \\	 
	\hhline {|:==::==::==:|}
	 $1$ & $0$ & $0$ & $1$ & $0$ & $1$ \\ \hhline {||--||--||--||}
	 $1$ & $0$ & $1$ & $0$ & $1$ & $0$ \\ \hhline {||--||--||--||}
	 $0$ & $1$ & $0$ & $1$ & $0$ & $1$ \\	
	\hhline {|:==::==::==:|}
	 $0$ & $0$ & $0$ & $1$ & $1$ & $0$ \\ \hhline {||--||--||--||}
	 $0$ & $0$ & $0$ & $1$ & $1$ & $0$ \\ \hhline {||--||--||--||}
	 $0$ & $0$ & $1$ & $0$ & $0$ & $1$ \\	
	\hhline {|b:==:b:==:b:==:b|}
	\end{tabular}
}
\ \ \
\subfloat[NAE-3SAT]{
\centering
	\begin{tabular}{||c|c||c|c||c|c||}
	\hhline {|t:==:t:==:t:==:t|}
	 $1$ & $0$ & $0$ & $1$ & $0$ & $0$ \\ \hhline {||--||--||--||}
	 $0$ & $1$ & $1$ & $0$ & $0$ & $0$ \\ \hhline {||--||--||--||}
	 $1$ & $1$ & $1$ & $1$ & $0$ & $0$ \\
	\hhline {|:==::==::==:|}
	 $1$ & $1$ & $0$ & $0$ & $0$ & $1$ \\ \hhline {||--||--||--||}
	 $1$ & $0$ & $0$ & $0$ & $1$ & $0$ \\ \hhline {||--||--||--||}
	 $0$ & $1$ & $0$ & $0$ & $1$ & $1$ \\	 
	\hhline {|:==::==::==:|}
	 $1$ & $0$ & $1$ & $1$ & $1$ & $1$ \\ \hhline {||--||--||--||}
	 $1$ & $1$ & $1$ & $0$ & $1$ & $0$ \\ \hhline {||--||--||--||}
	 $0$ & $1$ & $0$ & $1$ & $0$ & $1$ \\	
	\hhline {|:==::==::==:|}
	 $0$ & $0$ & $0$ & $1$ & $1$ & $0$ \\ \hhline {||--||--||--||}
	 $0$ & $0$ & $1$ & $1$ & $1$ & $1$ \\ \hhline {||--||--||--||}
	 $0$ & $0$ & $1$ & $0$ & $0$ & $1$ \\	
	\hhline {|b:==:b:==:b:==:b|}
	\end{tabular}
}
\caption{Examples of the objective vectors for MAX-3SAT, X3SAT, and NAE-3SAT problems}
\label {SATreduction}
\end{table}
\endgroup

With each truth assignment $u$ we associate a subset $Z(u)$ of $SATP(m,n)$ integral vertices, such that
$$\forall z \in Z(u): \ \textbf {\emph {row}}_{i}(z) = 1 - u_{i},$$
and $\textbf {\emph {col}}_{j}(z)$ can take any values.

Now we consider a linear objective function $f_{v} (x) =\langle v, x \rangle$. 

\begin {Theorem}
Maximum of the objective function $f_{v}(x)$ over $SATP(m,n)$ equals to the largest possible number of clauses that can be satisfied for MAX-3SAT problem.
\end {Theorem}
\begin{proof}
Let $v_{j}$ be the $j$-th column of blocks of the vector $v$.
By definition of vector $v$ for any column $v_{j}$ and for any integral vertex $z \in SATP(m,n)$ we have
$$\langle v_{j}, z_{j} \rangle \leq 1,$$
where $z_{j}$ is the $j$-th column of blocks of the vertex $z$.

It suffices to verify that if for some integral vertex $z \in SATP(m,n)$:
\begin {equation}
\langle v_{j}, z_{j} \rangle = 1,
\label {1column_eq}
\end {equation}
then the clause $c_{j}$ has at least one true literal on the corresponding truth assignment.

Now suppose that there exists a truth assignment $u$ that satisfy the clause $c_{j}$. Let $k$ be the position of a true literal in the clause $c_{j}$, then for any integral vertex $z \in Z(u)$, such that
$$\textbf {\emph {col}}_{j} (z) = k - 1,$$
the equality (\ref {1column_eq}) holds.
\end {proof}

A truth assignment can be reconstructed from the integral vertex $z$ that maximizes the objective function $f_{v}(x)$. Thus, MAX-3SAT is transformed to the integer programming over $SATP_{LP}(m,n)$ polytope. Similarly, we can consider weighted MAX-3SAT by multiplying the $j$-th column of the objective vector $v$ by the weight of the clause $c_{j}$.

For the different variants of 3-SAT we can slightly modify the objective function. For example, we consider one-in-three 3-satisfiability or exactly-1 3-satisfiability (X3SAT): given a set $U=\{u_{1}, \ldots, u_{m}\}$ of variables and a collection $C=\{c_{1}, \ldots, c_{n}\}$ of 3-literal clauses over $U$, the problem is to determine whether there exists a truth assignment to the variables so that each clause has exactly one true literal \cite {Schaefer}.

With each instance of the problem we associate an objective vector $w\in \R^{6mn}$:
\begin {itemize}
\item if a clause $c_{j}$ has a literal $u_{i}$ at the place $k$, then 
$$\forall s \in \{1,2,3\}\backslash k:\ w^{k,1}_{i,j} = w^{s,2}_{i,j} = 1,$$
\item if a clause $c_{j}$ has a literal $\overline{u}_{i}$ at the place $k$, then
$$\forall s \in \{1,2,3\}\backslash k:\ w^{k,2}_{i,j} = w^{s,1}_{i,j} = 1,$$
\item all the remaining coordinates of the vector $w$ equal to $0$.
\end {itemize} 

An example of an objective vector $w$ for the formula (\ref {boolean_formula}) is shown in Table \ref{SATreduction} (b).

We consider a linear objective function $f_{w} (x) =\langle w, x \rangle$. 

\begin {Theorem}
There exists a truth assignment for X3SAT problem with exactly one true literal per clause if and only if
$$\max_{x\in SATP_{LP}(m,n)} f_{w}(x) = \max_{z\in SATP(m,n)} f_{w}(z) = 3n.$$
\label {reduction_X3SAT}
\end {Theorem}
\begin {proof}
By definition of vector $w$ for any column of blocks $j$ and for any point $x \in SATP_{LP}(m,n)$ we have
$$\langle w_{j}, x_{j} \rangle \leq 3.$$
Thus, if $f_{w}(x) = 3n$, then $f_{w}(x_{j})=3$.

It suffices to verify that if for some integral vertex $z \in SATP(m,n)$:
\begin {equation}
\langle w_{j}, z_{j} \rangle = 3,
\label {3column_eq}
\end {equation}
then the clause $c_{j}$ has exactly one true literal on the corresponding truth assignment.

Now suppose that there exists a satisfying truth assignment $u$ for X3SAT problem. Let the clause $c_{j}$ have a true literal at the position $k$, then for any integral vertex $z \in Z(u)$, such that
$$\textbf {\emph {col}}_{j} (z) = k - 1,$$
the equality (\ref {3column_eq}) holds.
\end {proof}

Thus, X3SAT problem is transformed to the integer recognition over polytope $SATP_{LP}(m,n)$. Another popular variant of 3-SAT problem is Not-All-Equal 3-SAT (NAE-3SAT): given a set $U=\{u_{1}, \ldots, u_{m}\}$ of variables and a collection $C=\{c_{1}, \ldots, c_{n}\}$ of 3-literal clauses over $U$, the problem is to determine whether there exists a truth assignment so that each clause has at least one true literal and at least one false literal \cite {Schaefer}.

With each instance of the problem we associate an objective vector $y\in \R^{6mn}$:
\begin {itemize}
\item if a clause $c_{j}$ has a literal $u_{i}$ at the place $k$, then 
$$y^{k,1}_{i,j} = y^{(k+1) mod 3,2}_{i,j} = y^{(k+2) mod 3,1}_{i,j} = y^{(k+2) mod 3,2}_{i,j} = 1,$$
\item if a clause $c_{j}$ has a literal $\overline{u}_{i}$ at the place $k$, then
$$y^{k,2}_{i,j} = y^{(k+1) mod 3,1}_{i,j} = y^{(k+2) mod 3,1}_{i,j} = y^{(k+2) mod 3,2}_{i,j} = 1,$$
\item all the remaining coordinates of the vector $y$ are equal to $0$.
\end {itemize}

An example of an objective vector $y$ for the formula (\ref {boolean_formula}) is shown in Table \ref{SATreduction} (c).

We consider a linear objective function $f_{y} (x) =\langle y, x \rangle$. 

\begin {Theorem}
There exists a truth assignment for NAE-3SAT problem with at least one true literal and at least one false literal per clause if and only if
$$\max_{x\in SATP_{LP}(m,n)} f_{y}(x) = \max_{z\in SATP(m,n)} f_{y}(z) = 3n.$$
\end {Theorem}
\begin {proof}
The same as for X3SAT, it is just sufficient to replace vector $w$ with $y$.
\end {proof}

%%%%%%%%%%%%%%%%%%%%%%%%%%%%%%%%%%%%%%%%%%

\section {SATP 1-skeleton}

We will use the vectors $\textbf {\emph {row}} (z)$ and $\textbf {\emph {col}} (z)$ to establish the properties of $SATP(m,n)$ 1-skeleton.

\begin {Theorem}
Two vertices $u$ and $v$ of the $SATP(m,n)$ polytope are adjacent if and only if one of following conditions is true:
\begin {itemize}
\item $\textbf {row} (u) \neq \textbf {row} (v)$ and $\textbf {col} (u) \neq \textbf  {col} (v)$;
\item $\exists! i$: $\textbf {row}_{i} (u) \neq \textbf {row}_{i} (v)$ and $\textbf {col} (u) = \textbf {col} (v)$;
\item $\exists! j$: $\textbf {col}_{j} (u) \neq \textbf {col}_{j} (v)$ and $\textbf {row} (u) = \textbf {row} (v)$.
\end {itemize}
\label {theorem_adjacency}
\end {Theorem}

\begin {proof}
If the vertices $u$ and $v$ are not adjacent, then their convex hull intersects the convex hull of all the remaining vertices, and we have
\begin {equation}
\alpha u + (1-\alpha) v = \sum \lambda(w) w
\label {eq_convex_uvw}
\end {equation}
for some $\alpha, \lambda(w) \geq 0$, where $\sum \lambda (w) = 1$ and $w$ is an integral vertex of $SATP(m,n)$ other than $u$ and $v$.

We consider some vertex $w$ in equation (\ref {eq_convex_uvw}) with a positive $\lambda (w)$. Since $u,v$ and $w$ are zero-one points, equation implies the inequality
\begin {equation}
w \leq u + v.
\label {ineq_convex_uvw}
\end {equation}

Let the $\textbf {\emph {row}}$ and $\textbf {\emph {col}}$ vectors of $u$ and $v$ do not coincide. Since the vertices $u,v$ and $w$ are different, we have $\textbf {\emph {row}} (w) \neq \textbf {\emph {row}} (u)$ or $\textbf {\emph {row}} (w) \neq \textbf {\emph {row}} (v)$, and $\textbf {\emph {col}} (w) \neq \textbf {\emph {col}} (u)$ or $\textbf {\emph {col}} (w) \neq \textbf {\emph {col}} (v)$. Without loss of generality, we assume that
\begin{gather*}
	\exists i: \ \textbf {\emph {row}}_{i} (w) = 0 \neq \textbf {\emph {row}}_{i} (u),\\
	\exists j: \ \textbf {\emph {col}}_{j} (w) = 0 \neq \textbf {\emph {col}}_{j} (v).
\end{gather*}
Consequently, we have
\begin{align*}
	x^{1,1}_{i,j} (w) &= \frac {1}{2} (1-\textbf{\emph {row}}_{i}(w)) (2-\textbf{\emph {col}}_{j}(w)) (1-\textbf{\emph {col}} _{j}(w)) = 1,\\
	x^{1,1}_{i,j} (u) + x^{1,1}_{i,j} (w) &= \frac {1}{2} (1-\textbf{\emph {row}}_{i}(u)) (2-\textbf{\emph {col}}_{j}(u)) (1-\textbf{\emph {col}} _{j}(v)) + \\
	&+ \frac {1}{2} (1-\textbf{\emph {row}}_{i}(v)) (2-\textbf{\emph {col}}_{j}(v)) (1-\textbf{\emph {col}} _{j}(v)) = 0.
\end{align*}
Thus, the inequality (\ref {ineq_convex_uvw}) is not satisfied and the vertices $u$ and $v$ are adjacent. The remaining cases are treated similarly.

Let $\textbf {\emph {col}} (u)$ and $\textbf {\emph {col}} (v)$ be equal. We suppose that
$$\exists i,j: \ \textbf {\emph {row}}_{i} (u) \neq \textbf {\emph {row}}_{i} (v) \ \mbox{and} \ \textbf {\emph {row}}_{j} (u) \neq \textbf {\emph {row}}_{j} (v).$$

We consider two vertices $w_{u}$ and $w_{v}$, constructed as follows:
\begin{gather*}
	\textbf {\emph {col}} (w_{u}) = \textbf {\emph {col}} (w_{v}) = \textbf {\emph {col}} (u) = \textbf {\emph {col}} (v),\\
	\forall k (k \neq i): \textbf {\emph {row}}_{k} (w_{u}) = \textbf {\emph {row}}_{k} (u),\ \textbf {\emph {row}}_{k} (w_{v}) = \textbf {\emph {row}}_{k} (v),\\
	\textbf {\emph {row}}_{i} (w_{u}) = \textbf {\emph {row}}_{i} (v), \ \textbf{\emph {row}}_{i} (w_{v}) = \textbf {\emph {row}}_{i} (u).
\end{gather*}
Vertices $w_{u}$ and $w_{v}$ are different from $u$ and $v$, and we have
$$w_{u} + w_{v} = u + v,$$
thus, $u$ and $v$ are not adjacent.

Finally, if the vectors $\textbf {\emph {row}} (u)$ and $\textbf {\emph {row}} (v)$ differ only in one coordinate, then there are no vertices $w$ other than $u$ and $v$ that satisfy the inequality (\ref {ineq_convex_uvw}). Consequently, in this case $u$ and $v$ are adjacent as well.

The situation with the vectors $\textbf {\emph {row}} (u)$ and $\textbf {\emph {row}} (v)$ being equal should be treated in a similar way.
\end {proof}

Thus, 1-skeleton of $SATP(m,n)$ is not a complete graph, unlike $BQP(n)$. Still, this graph is very dense.

\begin {Corollary}
The diameter of $SATP(m,n)$ 1-skeleton equals 2.
The clique number of $SATP(m,n)$ 1-skeleton is superpolynomial in dimension and bounded from below by $2^{\min \{m,n\}}$.
\end {Corollary}

\begin {proof}
Let the vertices $u$ and $v$ of $SATP(m,n)$ polytope be not adjacent. We consider a vertex $w$, constructed as follows:
\begin{gather*}
	\textbf {\emph {row}}(w) \neq \textbf {\emph {row}} (u),\ \textbf {\emph {row}}(w) \neq \textbf {\emph {row}} (v),\\
	\textbf {\emph {col}}(w) \neq \textbf {\emph {col}} (u),\ \textbf {\emph {col}}(w) \neq \textbf {\emph {col}} (v).
\end{gather*}
By the assumption of Theorem \ref {theorem_adjacency}, we have $w$ being adjacent both to $u$ and $v$. Thus, the diameter of $SATP(m,n)$ equals 2. It is impossible to construct the non-adjacent points only if $m=n=1$. In this case all $6$ vertices are pairwise adjacent.

In order to prove a superpolynomial lower bound for the clique number we consider a vertex set $W$, such that $\forall w \in W$ we have
$$\forall k (k \leq \min \{m,n\}):\ \textbf {\emph {row}}_{k}(w) = \textbf {\emph {col}}_{k}(w).$$
All the remaining coordinates of $\textbf {\emph {row}}$ and $\textbf {\emph {col}}$ vectors are assumed to be zero. By Theorem \ref {theorem_adjacency}, each pair of vertices in $W$ is pairwise adjacent, since their $\textbf {\emph {row}}$ and $\textbf {\emph {col}}$ vectors do not coincide, and there are exactly $2^{\min \{m,n\}}$ of such vertices.
\end {proof}

In the last theorem of this section we will show that the properties of $SATP(m,n)$ 1-skeleton can be transferred to its LP relaxation $SATP_{LP}(m,n)$.

\begin {Theorem}
$SATP_{LP}(m,n)$ has the Trubin-property with respect to $SATP(m,n)$.
\end {Theorem}

\begin {proof}

Trubin \cite {Trubin} (see also \cite {Yemelichev}) showed that the relaxation set partitioning polytope
\begin {equation}
Ax=e, \ x \geq 0,
\label {part_polytope}
\end {equation}
where $A$ is a zero-one matrix, and $e$ is an all unit column, is quasi-integral. $SATP_{LP}(m,n)$ can be considered as a special case of the relaxation set partitioning polytope. Constraints (\ref {SATeqfirst}) and (\ref {SATineq}) already satisfy (\ref {part_polytope}), while the constraints (\ref {SATeqsecond})-(\ref {SATeqthird}) can easily be rewritten in the required form:
\begin {gather*}
\left\{
\begin {array} {l}
x^{1,1}_{i,j} + x^{2,1}_{i,j} + x^{3,1}_{i,j} = x^{1,1}_{i,t} + x^{2,1}_{i,t} + x^{3,1}_{i,t}, \\
x^{1,1}_{i,j} + x^{2,1}_{i,j} + x^{3,1}_{i,j} + x^{1,2}_{i,j} + x^{2,2}_{i,j} + x^{3,2}_{i,j} = 1
\end {array}
\right.
\ \Rightarrow \ 
\left\{
\begin {array} {l}
x^{1,1}_{i,t} + x^{2,1}_{i,t} + x^{3,1}_{i,t} + x^{1,2}_{i,j} + x^{2,2}_{i,j} + x^{3,2}_{i,j} = 1,\\
x^{1,1}_{i,j} + x^{2,1}_{i,j} + x^{3,1}_{i,j} + x^{1,2}_{i,j} + x^{2,2}_{i,j} + x^{3,2}_{i,j} = 1
\end {array}
\right. \\
\left\{
\begin {array} {l}
x^{1,1}_{i,j} + x^{1,2}_{i,j} = x^{1,1}_{s,j} + x^{1,2}_{s,j}, \\
x^{1,1}_{i,j} + x^{1,2}_{i,j} + x^{2,1}_{i,j} + x^{2,2}_{i,j} + x^{3,1}_{i,j} + x^{3,2}_{i,j} = 1
\end {array}
\right.
\ \Rightarrow \ 
\left\{
\begin {array} {l}
x^{1,1}_{s,j} + x^{1,2}_{s,j} + x^{2,1}_{i,j} + x^{2,2}_{i,j} + x^{3,1}_{i,j} + x^{3,2}_{i,j} = 1,\\
x^{1,1}_{i,j} + x^{1,2}_{i,j} + x^{2,1}_{i,j} + x^{2,2}_{i,j} + x^{3,1}_{i,j} + x^{3,2}_{i,j} = 1
\end {array}
\right.
\end {gather*}

Thus, 1-skeleton of $SATP(m,n)$ is a subset of 1-skeleton of $SATP_{LP}(m,n)$.
\end {proof}

%%%%%%%%%%%%%%%%%%%%%%%%%%%%%%%%%%%%%%%%%%

\section {Fractional vertices}

Now we consider the polytope $SATP_{LP}(m,n)$. It preserves all integral vertices of $SATP(m,n)$, together with their adjacency relationships, but as LP relaxation has its own fractional vertices. In this section, we will see that their properties are much closer to the metric polytope $MET (n)$ than to $BQP_{LP}(n)$.

Fractional vertices of $BQP_{LP}(n)$ are quite simple with values only from the set $\{0,\frac{1}{2},1\}$. Thereby, they can be completely cut off by triangle inequality constraints of the metric polytope \cite {Padberg}. Integer recognition over $BQP_{LP}(n)$ is polynomially solvable based on this fact (Lemma \ref {lemma_cut_BQPLP} and Theorem \ref {theorem_BQPLP_polynomial} \cite {Bondarenko-Uryvaev}).

Fractional vertices of the metric polytope $MET(n)$ have a much more complicated nature. Their denominators can take any integral values (Theorem \ref {theorem_MET_vertices} \cite {Laurent}) and grow exponentially \cite {Nikolaev}. It is not known if it is possible to cut them off by a polynomial number of additional linear constraints \cite {Bondarenko-NSS}. Furthermore, characteristics of $MET(n)$ fractional vertices were also considered in \cite {Avis, Deza, Grishukhin}. 

Unfortunately, properties of $SATP_{LP}(m,n)$ fractional vertices are closer to $MET(n)$, as their denominators can take any integral values as well.

\begin {Theorem}
The relaxation polytope $SATP_{LP}(n,n)$ has fractional vertices with denominators equal $n+1$ for all $n \geq 4$.
\label {theorem_denominators}
\end {Theorem}

\begin {proof}
A vertex of $SATP_{LP}(m,n)$ polytope is a unique solution of the system (\ref {SATeqfirst})-(\ref {SATineq}) with some of inequalities (\ref {SATineq}) turned into equations. 
We construct a required vertex in a few steps. The basis is the first four blocks as shown in Table \ref {table_fractional_basis}.
\begin{table}[h]
\centering
\begin {tabular} {||c|c||c|c||}
\hhline {|t:==:t:==t:|}
$x^{1,1}_{1,1}$ & $0$ & $x^{1,1}_{1,2}$ & $0$ \\ \hhline {||--||--||}
$x^{2,1}_{1,1}$ & $0$ & $0$ & $x^{2,2}_{1,2}$ \\ \hhline {||--||--||}
$0$ & $x^{3,2}_{1,1}$ & $x^{3,1}_{1,2}$ & $0$ \\ \hhline {|:==::==:|}

$x^{1,1}_{2,1}$ & $0$ & $x^{1,1}_{2,2}$ & $0$ \\ \hhline {||--||--||}
$x^{2,1}_{2,1}$ & $0$ & $x^{2,1}_{2,2}$ & $0$ \\ \hhline {||--||--||}
$0$ & $x^{3,2}_{2,1}$ & $0$ & $x^{3,2}_{2,2}$ \\ \hhline {|b:==:b:==:b|}
\end {tabular}
\caption{First four blocks of the fractional vertex}
\label {table_fractional_basis}
\end {table}

We can use the constraints (\ref {SATeqsecond})-(\ref {SATeqthird}) to establish the relationship between the coordinates:
$$x^{2,2}_{1,2} = x^{3,2}_{1,1} = x^{3,2}_{2,1} = x^{3,2}_{2,2} = x^{3,1}_{1,2}.$$
Hence, for all blocks in the second column
$$x^{2,1}_{i,2} + x^{2,2}_{i,2} = x^{3,1}_{i,2} + x^{3,2}_{i,2}.$$

Here, we describe the key steps of the construction.
For all $j$ ($1\leq j \leq n-1$) blocks $j,j$ and $j,j+1$ have the form as shown in Table \ref {table_fractional_j}.
\begin{table}[h]
\centering
\begin {tabular} {||c|c||c|c||}
\hhline {|t:==:t:==t:|}
$x^{1,1}_{j,j}$ & $0$ & $x^{1,1}_{j,j+1}$ & $0$ \\ \hhline {||--||--||}
$x^{2,1}_{j,j}$ & $0$ & $0$ & $x^{2,2}_{j,j+1}$ \\ \hhline {||--||--||}
$0$ & $x^{3,2}_{j,j}$ & $x^{3,1}_{j,j+1}$ & $0$ \\  \hhline {|b:==:b:==:b|}
\end {tabular}
\caption{Blocks $j,j$ and $j,j+1$ of the fractional vertex.}
\label {table_fractional_j}
\end {table}
Thus, for all $i,j$ we have 
\begin {equation}
x^{3,1}_{i,j} + x^{3,2}_{i,j} = x^{2,1}_{i,j+1} + x^{2,2}_{i,j+1}.
\label {eq_block_j}
\end {equation}

For all $k$ ($2\leq k \leq \left\lfloor \frac{n}{2} \right\rfloor)$ there are blocks in the rows $2k-1$ and $2k$ as shown in Table \ref {table_fractional_2k}. If $2k>n$, the last row and column can be omitted.
\begin{table}[h]
\centering
\begin {tabular} {||c|c||c|c||c|c||}
\hhline {|t:==:t:==:t:==t:|}
$x^{1,1}_{2k-1,k}$ & $0$ & $x^{1,1}_{2k-1,2k-1}$ & $0$ & - & - \\ \hhline {||--||--||--||}
$0$ & $x^{2,2}_{2k-1,k}$ & $x^{2,1}_{2k-1,2k-1}$ & $0$ & - & - \\ \hhline {||--||--||--||}
$0$ & $x^{3,2}_{2k-1,k}$ & $0$ & $x^{3,2}_{2k-1,2k-1}$ & - & - \\ \hhline {|:==::==::==:|}

$x^{1,1}_{2k,k}$ & $0$ & - & - & $x^{1,1}_{2k,2k}$ & $0$  \\ \hhline {||--||--||--||}
$0$ & $x^{2,2}_{2k,k}$ & - & - & $x^{2,1}_{2k,2k}$ & $0$  \\ \hhline {||--||--||--||}
$0$ & $x^{3,2}_{2k,k}$ & - & - & $0$ & $x^{3,2}_{2k,2k}$ \\ \hhline {|b:==:b:==:b:==:b|}
\end {tabular}
\caption{Fragment of $2k-1$ and $2k$ rows of blocks of the fractional vertex}
\label {table_fractional_2k}
\end {table}

Here, we obtain
\begin {equation}
x^{3,1}_{i,2k-1} + x^{3,2}_{i,2k-1} = x^{3,1}_{i,2k} + x^{3,2}_{i,2k} = x^{2,1}_{i,k} + x^{2,2}_{i,k} + x^{3,1}_{i,k} + x^{3,2}_{i,k}
\label {eq_block_k}
\end {equation}
for all blocks in these columns.

The last part of construction describes the blocks in the rows $n-1$ and $n$, as shown in Table \ref {table_fractional_n}.
\begin{table}[h]
\centering
\begin {tabular} {||c|c||c|c||}
\hhline {|t:==:t:==t:|}
$0$ & $x^{1,1}_{n-1,1}$ & $x^{1,1}_{n-1,n}$ & $0$ \\ \hhline {||--||--||}
$x^{2,1}_{n-1,1}$ & $0$ & $0$ & $x^{2,2}_{n-1,n}$ \\ \hhline {||--||--||}
$x^{3,1}_{n-1,1}$ & $0$ & $x^{3,1}_{n-1,n}$ & $0$ \\ \hhline {|:==::==:|}

$x^{1,1}_{2,1}$ & $0$ & $x^{1,1}_{n,n}$ & $0$ \\ \hhline {||--||--||}
$0$ & $x^{2,2}_{n,1}$ & $x^{2,1}_{n,n}$ & $0$ \\ \hhline {||--||--||}
$x^{3,1}_{n,1}$ & $0$ & $0$ & $x^{3,2}_{n,n}$ \\ \hhline {|b:==:b:==:b|}
\end {tabular}
\caption{Last two rows of blocks of the fractional vertex}
\label {table_fractional_n}
\end {table}

Hence, for blocks in the first and last columns we have
\begin{align}
x^{1,1}_{i,1} + x^{1,2}_{i,1} &= x^{2,1}_{i,n} + x^{2,2}_{i,n}, \label {eq_block_n_first} \\
x^{2,1}_{i,1} + x^{2,2}_{i,1} &= x^{3,1}_{i,n} + x^{3,2}_{i,n} \label {eq_block_n_second}.
\end {align}

It is possible to make last two rows different from the first two rows since $n \geq 4$.

We call all of the remaining blocks that were not described in the preceding steps as filler blocks. They are different for the blocks above and below the main diagonal and have the form as shown in Table \ref {table_fractional_filler}.
\begin{table}[h]
\centering
\subfloat[]{
\begin {tabular} {||c|c||}
\hhline {|t:==t:|}
$x^{1,1}_{i,j}$ & $0$ \\ \hhline {||--||}
$x^{2,1}_{i,j}$ & $0$ \\ \hhline {||--||}
$x^{3,1}_{i,j}$ & $x^{3,2}_{i,j}$ \\ \hhline {|b:==:b|}
\end {tabular}
}
\ \ \ 
\subfloat []{
\begin {tabular} {||c|c||}
\hhline {|t:==t:|}
$x^{1,1}_{i,j}$ & $x^{1,1}_{i,j}$ \\ \hhline {||--||}
$x^{2,1}_{i,j}$ & $0$ \\ \hhline {||--||}
$0$ & $x^{3,2}_{i,j}$ \\ \hhline {|b:==:b|}
\end {tabular}
}
\caption{Form of the filler blocks of the fractional vertex above the main diagonal (A) and beyond the main diagonal (B)}
\label {table_fractional_filler}
\end {table}

Now we will show that the system (\ref {SATeqfirst})-(\ref {SATineq}) with such zero variables, described above, has a unique solution. 
Let $n$ be odd and equal $2q+1$. We denote $x^{3,2}_{1,1}$ simply as $x$. 
Thereby, for all $i$ by equations (\ref {eq_block_j})-(\ref {eq_block_k}) and induction we get
\begin{align*}
x^{2,1}_{i,2} + x^{2,2}_{i,2} &= x, \\
x^{3,1}_{i,2} + x^{3,2}_{i,2} &= x, \\
x^{2,1}_{i,3} + x^{2,2}_{i,3} &= x, \\
x^{3,1}_{i,3} + x^{3,2}_{i,3} &= 2x, \\
\ldots\\
x^{2,1}_{i,2q} + x^{2,2}_{i,2q} &= q x, \\
x^{3,1}_{i,2q} + x^{3,2}_{i,2q} &= q x, \\
x^{2,1}_{i,2q+1} + x^{2,2}_{i,2q+1} &= q x, \\
x^{3,1}_{i,2q+1} + x^{3,2}_{i,2q+1} &= (q+1) x.
\end {align*}

First and last columns are connected by equations (\ref {eq_block_n_first})-(\ref {eq_block_n_second}), therefore 
\begin{align*}
x^{1,1}_{1,1} + x^{1,2}_{1,1} &= q x, \\
x^{2,1}_{1,1} + x^{2,2}_{1,1} &= (q+1) x,\\
x^{3,1}_{1,1} + x^{3,2}_{1,1} &= x.
\end {align*}

Since the sum of the coordinates inside a single block is equal to one, we have 
$$q x + (q+1) x + x = 1,$$
and
\begin {equation}
x = \frac {1}{2q+2} = \frac {1}{n+1}.
\label {xvalue}
\end {equation}
All coordinates of the constructed point are either already directly expressed in terms of $x$, or can be found using the equations (\ref {SATeqfirst})-(\ref {SATeqthird}). Thus, it is a unique solution of the system (\ref {SATeqfirst})-(\ref {SATineq}) and a vertex of the polytope $SATP_{LP}(n,n)$ with a denominator $n+1$.

Case of $n$ equal $2q$ is considered similarly, the only difference will be that
\begin{align*}
x^{1,1}_{1,1} + x^{1,2}_{1,1} &= x^{2,1}_{i,2q} + x^{2,2}_{i,2q} = q x, \\
x^{2,1}_{1,1} + x^{2,2}_{1,1} &= x^{3,1}_{i,2q} + x^{3,2}_{i,2q} = q x.
\end {align*}

It remains to verify only that the coordinates of the filler blocks from the Table \ref {table_fractional_filler} satisfy the system (\ref {SATeqfirst})-(\ref {SATineq}). 
We consider the filler blocks above the main diagonal ($i<j$). Using equations (\ref {SATeqfirst})-(\ref {SATeqthird}) we can establish that
\begin{align*}
x^{2,1}_{i,j} &= \left\lfloor \frac {j}{2} \right\rfloor x, \\
x^{3,2}_{i,j} &= \left\lfloor \frac {i+1}{2} \right\rfloor x, \\
x^{3,1}_{i,j} &= \left\lfloor \frac {j+1}{2} \right\rfloor x - \left\lfloor \frac {i+1}{2} \right\rfloor x, \\
x^{1,1}_{i,j} &= 1 - \left\lfloor \frac {j}{2} \right\rfloor x - \left\lfloor \frac {j+1}{2} \right\rfloor x.
\end {align*}
Hence, only the inequalities $x^{3,1}_{i,j} \geq 0$ and $x^{1,1}_{i,j} \geq 0$ can be violated. For all $i<j$ we have
$$\left\lfloor \frac {j+1}{2} \right\rfloor \geq \left\lfloor \frac {i+1}{2} \right\rfloor.$$
Therefore, $x^{3,1}_{i,j} \geq 0$. And, since $j \leq n$, we have
$$\left\lfloor \frac {j}{2} \right\rfloor + \left\lfloor \frac {j+1}{2} \right\rfloor < n + 1.$$
Thus, by (\ref {xvalue}), $x^{1,1}_{i,j} \geq 0$ is satisfied as well.

Now we consider the filler blocks below the main diagonal ($i>j$):
\begin{align*}
x^{2,1}_{i,j} &= \left\lfloor \frac {j}{2} \right\rfloor x, \\
x^{3,2}_{i,j} &= \left\lfloor \frac {j+1}{2} \right\rfloor x, \\
x^{1,2}_{i,j} &= \left\lfloor \frac {i+1}{2} \right\rfloor x - \left\lfloor \frac {j+1}{2} \right\rfloor x, \\
x^{1,1}_{i,j} &= 1 - \left\lfloor \frac {i+1}{2} \right\rfloor x - \left\lfloor \frac {j}{2} \right\rfloor x.
\end {align*}
Again, for all $i>j$ we have
$$\left\lfloor \frac {i+1}{2} \right\rfloor \geq \left\lfloor \frac {j+1}{2} \right\rfloor,$$
and the inequality $x^{1,2}_{i,j} \geq 0$ is satisfied. And, since $i,j \leq n$:
$$\left\lfloor \frac {j}{2} \right\rfloor + \left\lfloor \frac {i+1}{2} \right\rfloor < n + 1.$$
Thus, $x^{1,1}_{i,j} \geq 0$ holds as well.

The constructed system is obtained from (\ref {SATeqfirst})-(\ref {SATineq}) by turning some of inequalities into equations, and it has a unique solution, therefore it defines the fractional vertex of the $SATP_{LP}(n,n)$ polytope with denominator $n+1$.
\end {proof}

An example of a fractional vertex for $n = 6$ is shown in Table \ref {vertex_example}.
\begin{table}[p]
\centering
\begin {tabular} {||c|c||c|c||c|c||c|c||c|c||c|c||}
\hhline {|t:==:t:==:t:==:t:==:t:==:t:==:t|}

$\frac{3}{7}$ & $0$ & $\frac{5}{7}$ & $0$ & $\frac{4}{7}$ & $0$ & $\frac{3}{7}$ & $0$ & $\frac{2}{7}$ & $0$ & $\frac{1}{7}$ & $0$ \\ \hhline {||--||--||--||--||--||--||}
$\frac{3}{7}$ & $0$ & $0$ & $\frac{1}{7}$ & $\frac{1}{7}$ & $0$ & $\frac{2}{7}$ & $0$ & $\frac{2}{7}$ & $0$ & $\frac{3}{7}$ & $0$ \\ \hhline {||--||--||--||--||--||--||}
$0$ & $\frac{1}{7}$ & $\frac{1}{7}$ & $0$ & $\frac{1}{7}$ & $\frac{1}{7}$ & $\frac{1}{7}$ & $\frac{1}{7}$ & $\frac{2}{7}$ & $\frac{1}{7}$ & $\frac{2}{7}$ & $\frac{1}{7}$ \\ \hhline {|:==::==::==::==::==::==:|}

$\frac{3}{7}$ & $0$ & $\frac{5}{7}$ & $0$ & $\frac{4}{7}$ & $0$ & $\frac{3}{7}$ & $0$ & $\frac{2}{7}$ & $0$ & $\frac{1}{7}$ & $0$ \\ \hhline {||--||--||--||--||--||--||}
$\frac{3}{7}$ & $0$ & $\frac{1}{7}$ & $0$ & $0$ & $\frac{1}{7}$ & $\frac{2}{7}$ & $0$ & $\frac{2}{7}$ & $0$ & $\frac{3}{7}$ & $0$ \\ \hhline {||--||--||--||--||--||--||}
$0$ & $\frac{1}{7}$ & $0$ & $\frac{1}{7}$ & $\frac{2}{7}$ & $0$ & $\frac{1}{7}$ & $\frac{1}{7}$ & $\frac{2}{7}$ & $\frac{1}{7}$ & $\frac{2}{7}$ & $\frac{1}{7}$ \\ \hhline {|:==::==::==::==::==::==:|}

$\frac{2}{7}$ & $\frac{1}{7}$ & $\frac{5}{7}$ & $0$ & $\frac{4}{7}$ & $0$ & $\frac{3}{7}$ & $0$ & $\frac{2}{7}$ & $0$ & $\frac{1}{7}$ & $0$ \\ \hhline {||--||--||--||--||--||--||}
$\frac{3}{7}$ & $0$ & $0$ & $\frac{1}{7}$ & $\frac{1}{7}$ & $0$ & $0$ & $\frac{2}{7}$ & $\frac{2}{7}$ & $0$ & $\frac{3}{7}$ & $0$ \\ \hhline {||--||--||--||--||--||--||}
$0$ & $\frac{1}{7}$ & $0$ & $\frac{1}{7}$ & $0$ & $\frac{2}{7}$ & $\frac{2}{7}$ & $0$ & $\frac{1}{7}$ & $\frac{2}{7}$ & $\frac{1}{7}$ & $\frac{2}{7}$ \\ \hhline {|:==::==::==::==::==::==:|}

$\frac{2}{7}$ & $\frac{1}{7}$ & $\frac{5}{7}$ & $0$ & $\frac{4}{7}$ & $0$ & $\frac{3}{7}$ & $0$ & $\frac{2}{7}$ & $0$ & $\frac{1}{7}$ & $0$ \\ \hhline {||--||--||--||--||--||--||}
$\frac{3}{7}$ & $0$ & $0$ & $\frac{1}{7}$ & $\frac{1}{7}$ & $0$ & $\frac{2}{7}$ & $0$ & $0$ & $\frac{2}{7}$ & $\frac{3}{7}$ & $0$ \\ \hhline {||--||--||--||--||--||--||}
$0$ & $\frac{1}{7}$ & $0$ & $\frac{1}{7}$ & $0$ & $\frac{2}{7}$ & $0$ & $\frac{2}{7}$ & $\frac{3}{7}$ & $0$ & $\frac{1}{7}$ & $\frac{2}{7}$ \\ \hhline {|:==::==::==::==::==::==:|}

$0$ & $\frac{3}{7}$ & $\frac{3}{7}$ & $\frac{2}{7}$ & $\frac{4}{7}$ & $0$ & $\frac{2}{7}$ & $\frac{1}{7}$ & $\frac{2}{7}$ & $0$ & $\frac{1}{7}$ & $0$ \\ \hhline {||--||--||--||--||--||--||}
$\frac{3}{7}$ & $0$ & $\frac{1}{7}$ & $0$ & $0$ & $\frac{1}{7}$ & $\frac{2}{7}$ & $0$ & $\frac{2}{7}$ & $0$ & $0$ & $\frac{3}{7}$ \\ \hhline {||--||--||--||--||--||--||}
$\frac{1}{7}$ & $0$ & $0$ & $\frac{1}{7}$ & $0$ & $\frac{2}{7}$ & $0$ & $\frac{2}{7}$ & $0$ & $\frac{3}{7}$ & $\frac{3}{7}$ & $0$ \\ \hhline {|:==::==::==::==::==::==:|}

$\frac{3}{7}$ & $0$ & $\frac{3}{7}$ & $\frac{2}{7}$ & $\frac{4}{7}$ & $0$ & $\frac{2}{7}$ & $\frac{1}{7}$ & $\frac{2}{7}$ & $0$ & $\frac{1}{7}$ & $0$ \\ \hhline {||--||--||--||--||--||--||}
$0$ & $\frac{3}{7}$ & $\frac{1}{7}$ & $0$ & $0$ & $\frac{1}{7}$ & $\frac{2}{7}$ & $0$ & $\frac{2}{7}$ & $0$ & $\frac{3}{7}$ & $0$ \\ \hhline {||--||--||--||--||--||--||}
$\frac{1}{7}$ & $0$ & $0$ & $\frac{1}{7}$ & $0$ & $\frac{2}{7}$ & $0$ & $\frac{2}{7}$ & $0$ & $\frac{3}{7}$ & $0$ & $\frac{3}{7}$ \\ \hhline {|b:==:b:==:b:==:b:==:b:==:b:==:b|}
\end {tabular}
\caption{Fractional vertex of the polytope $SATP_{LP}(6,6)$}
\label {vertex_example}
\end {table}

The construction described in Theorem \ref {theorem_denominators} is not working for $n<4$. However, relaxation polytope $SATP_{LP}(m,n)$ has fractional vertices with denominators $2,3$ and $4$ as well. Some examples are provided in Table \ref {vertex_examples_small}. It may be noted that Theorem \ref{theorem_denominators} holds for $n = 3$, but not for $n = 1$ and $n = 2$. Polytope $SATP_{LP}(1,1)$ coincide with $SATP(1,1)$ and has only $6$ integral vertices. Polytope $SATP_{LP}(2,2)$ has $72$ fractional vertices that can be computed. All of them have denominators equal to $2$ \cite {Skeleton}.
\begin{table}[p]
\centering
\begin {tabular} {||c|c||c|c||}
\hhline {|t:==:t:==t:|}
$\frac{1}{2}$ & $0$ & $0$ & $\frac{1}{2}$ \\ \hhline {||--||--||}
$0$ & $\frac{1}{2}$ & $\frac{1}{2}$ & $0$ \\ \hhline {||--||--||}
$0$ & $0$ & $0$ & $0$ \\ \hhline {|:==::==:|}

$\frac{1}{2}$ & $0$ & $\frac{1}{2}$ & $0$ \\ \hhline {||--||--||}
$0$ & $\frac{1}{2}$ & $0$ & $\frac{1}{2}$ \\ \hhline {||--||--||}
$0$ & $0$ & $0$ & $0$ \\ \hhline {|b:==:b:==:b|}
\end {tabular}
\ \ \
\begin {tabular} {||c|c||c|c||}
\hhline {|t:==:t:==t:|}
$0$ & $\frac{1}{3}$ & $\frac{2}{3}$ & $0$ \\ \hhline {||--||--||}
$\frac{1}{3}$ & $0$ & $0$ & $\frac{1}{3}$ \\ \hhline {||--||--||}
$\frac{1}{3}$ & $0$ & $0$ & $0$ \\ \hhline {|:==::==:|}

$\frac{1}{3}$ & $0$ & $\frac{2}{3}$ & $0$ \\ \hhline {||--||--||}
$0$ & $\frac{1}{3}$ & $0$ & $\frac{1}{3}$ \\ \hhline {||--||--||}
$\frac{1}{3}$ & $0$ & $0$ & $0$ \\ \hhline {|:==::==:|}

$\frac{1}{3}$ & $0$ & $\frac{2}{3}$ & $0$ \\ \hhline {||--||--||}
$\frac{1}{3}$ & $0$ & $0$ & $\frac{1}{3}$ \\ \hhline {||--||--||}
$0$ & $\frac{1}{3}$ & $0$ & $0$ \\ \hhline {|b:==:b:==:b|}
\end {tabular}
\ \ \
\begin {tabular} {||c|c||c|c||c|c||}
\hhline {|t:==:t:==:t:==t:|}
$\frac{1}{2}$ & $0$ & $\frac{1}{2}$ & $0$ & $\frac{1}{2}$ & $0$  \\ \hhline {||--||--||--|}
$\frac{1}{4}$ & $0$ & $\frac{1}{4}$ & $0$ & $\frac{1}{4}$ & $0$ \\ \hhline {||--||--||--|}
$0$ & $\frac{1}{4}$ & $0$ & $\frac{1}{4}$ & $0$ & $\frac{1}{4}$ \\ \hhline {|:==::==::==:|}

$\frac{1}{2}$ & $0$ & $\frac{1}{2}$ & $0$ & $\frac{1}{2}$ & $0$  \\ \hhline {||--||--||--|}
$0$ & $\frac{1}{4}$ & $\frac{1}{4}$ & $0$ & $0$ & $\frac{1}{4}$ \\ \hhline {||--||--||--|}
$\frac{1}{4}$ & $0$ & $0$ & $\frac{1}{4}$ & $\frac{1}{4}$ & $0$ \\ \hhline {|:==::==::==:|}

$\frac{1}{2}$ & $0$ & $\frac{1}{2}$ & $0$ & $0$ & $\frac{1}{2}$  \\ \hhline {||--||--||--|}
$0$ & $\frac{1}{4}$ & $0$ & $\frac{1}{4}$ & $\frac{1}{4}$ & $0$ \\ \hhline {||--||--||--|}
$0$ & $\frac{1}{4}$ & $0$ & $\frac{1}{4}$ & $\frac{1}{4}$ & $0$ \\ \hhline {|b:==:b:==:b:==:b|}
\end {tabular}
\caption{Fractional vertices with denominators $2,3$ and $4$}
\label {vertex_examples_small}
\end {table}

%%%%%%%%%%%%%%%%%%%%%%%%%%%%%%%%%%%%%%%%%%

\section {Integer recognition}

In this section we consider the problem of integer recognition. It is known that this problem is NP-complete over entire $SATP_{LP}(m,n)$ polytope (Theorem \ref {Integer_recognition_SATP_NPC}), but polynomially solvable over its face $BQP(n)$ (Theorem \ref {theorem_BQPLP_polynomial}). However, for some objective functions, other than those specified above in Section \ref {SAT_polytope_section}, integer recognition over $SATP_{LP}(m,n)$ can be solved efficiently.

We consider a vector $c\in \R^{6mn}$, such that
\begin {gather}
\nonumber
\forall j\in \N_{n}, \ \exists a,b \in \{1,2,3\}\ (a \neq b), \forall i\in \N_{m}:\\ 
c^{a,1}_{i,j}+c^{b,2}_{i,j}=c^{a,2}_{i,j}+c^{b,1}_{i,j},
\label {obj_vector_IR}
\end {gather}
and a corresponding linear objective function $f_{c} (x) =\langle c, x \rangle$.

\begin {Theorem}
For objective functions of the form $f_{c} (x)$ the problem of integer recognition over $SATP_{LP}(m,n)$ polytope is polynomially solvable.
\label {integer_recognition_SATP_polynomial}
\end {Theorem}

\begin {proof}
Without loss of generality, we assume that the vector $c$ has the form:
\begin {equation}
\forall i,j: c^{2,1}_{i,j} + c^{3,2}_{i,j} = c^{2,2}_{i,j} + c^{3,1}_{i,j}.
\label {obj_vector_IR_for_theorem}
\end {equation}
For any other choices of restrictions on vector $c$ following proof can be modified by just renaming the coordinates.

To make room for superscripts we introduce a new notation for the coordinates of the polytope:
\begin {gather*}
x^{1,1}_{i,j}=x_{i,j},\ \ x^{1,2}_{i,j}=y_{i,j},\ \ x^{2,1}_{i,j}=z_{i,j},\\
x^{2,2}_{i,j}=t_{i,j},\ \ x^{3,1}_{i,j}=u_{i,j},\ \ x^{3,2}_{i,j}=v_{i,j}.
\end {gather*}

We construct a new polytope $SATP^{2}_{LP}(m,n)$, satisfying the system (\ref {SATeqfirst})-(\ref {SATineq}) and the additional constraints
\begin {gather}
y_{i,j}+z_{i,j}+u_{i,j}+x_{i,l}+t_{i,l}+v_{i,l}+x_{k,j}+t_{k,j}+v_{k,j}+x_{k,l}+t_{k,l}+v_{k,l}\leq 3, \label {SATP2_ineq_first}\\
y_{i,j}+z_{i,j}+u_{i,j}+y_{i,l}+z_{i,l}+u_{i,l}+x_{k,j}+t_{k,j}+v_{k,j}+y_{k,l}+z_{k,l}+u_{k,l}\leq 3, \label {SATP2_ineq_second}
\end {gather}
for all $i,k \in \N_{m}$ $(i \neq k)$ and $j,l \in \N_{n}$ $(j \neq l)$.

All integral vertices of $SATP(m,n)$ (Lemma \ref {integral_vertices}) satisfy the inequalities (\ref{SATP2_ineq_first})-(\ref{SATP2_ineq_second}), therefore $SATP^{2}_{LP}(m,n)$ is another LP relaxation of $SATP(m,n)$ polytope.

Note that the total number of additional constraints (\ref{SATP2_ineq_first})-(\ref{SATP2_ineq_second}) is polynomially bounded above by $O(m^{2} n^{2})$.

Let $w$ be the point that maximize the function $f_{c} (x) $ over $SATP^{2}_{LP}(m,n)$. We claim that there exists a point $w^{*} \in SATP^{2}_{LP}(m,n)$ with $f_{c} (w) = f_{c} (w^{*})$ and $\forall i,j:$ $x^{w^{*}}_{i,j}>0$, up to renaming the coordinates. We construct $w^{*}$ from $w$ in a few steps.

\begin{enumerate}
\item \label{step_1} If there exists some $i$ that 
$$x^{w}_{i,j}+z^{w}_{i,j}+u^{w}_{i,j}=0,$$
then we change the columns in all blocks of the $i$-th row (Table \ref {columns_rearrange}).
\begin{table}[p]
\centering
\begin{tabular}{||c|c||c|c||}
\hhline {|t:==:t:==:t|}
$\ 0 \ $ & $y^{w}_{i,j}$ & $\ 0 \ $ & $y^{w}_{i,l}$ \\ \hhline {||--||--||}
$0$ & $t^{w}_{i,j}$ & $0$ & $t^{w}_{i,l}$ \\ \hhline {||--||--||}
$0$ & $v^{w}_{i,j}$ & $0$ & $v^{w}_{i,l}$ \\
\hhline {|b:==:b:==:b|}
\end{tabular}
$\Rightarrow$
\begin{tabular}{||c|c||c|c||}
\hhline {|t:==:t:==:t|}
$x^{w^{*}}_{i,j} = y^{w}_{i,j}$ & $\ 0 \ $ & $x^{w^{*}}_{i,l} = y^{w}_{i,l}$ & $\ 0 \ $ \\ \hhline {||--||--||}
$z^{w^{*}}_{i,j} = t^{w}_{i,j}$ & $0$ & $z^{w^{*}}_{i,l} = t^{w}_{i,l}$ & $0$ \\ \hhline {||--||--||}
$u^{w^{*}}_{i,j} = v^{w}_{i,j}$ & $0$ & $u^{w^{*}}_{i,l} = v^{w}_{i,l}$ & $0$ \\
\hhline {|b:==:b:==:b|}
\end{tabular}
\caption {Rearrange of the columns in the $i$-th row of the block matrix at Step \ref {step_1}}
\label {columns_rearrange}
\end{table}

Due to the symmetry of the system (\ref {SATeqfirst})-(\ref {SATineq}),(\ref{SATP2_ineq_first}),(\ref{SATP2_ineq_second}) and the constraints (\ref{obj_vector_IR_for_theorem}), the new point belongs to the polytope $SATP^{2}_{LP}(m,n)$ and has the same value of the objective function. In fact, we simply rename some coordinates. Thus, we can now consider $w^{*}$ simply as $w$ and continue the procedure.

\item \label {step_2} If there exists some $j$ that
$$z^{w}_{i,j}+t^{w}_{i,j}=0\ \mbox{and} \ u^{w}_{i,j}+v^{w}_{i,j}>0,$$
then we change the second and third rows in all blocks of the $j$-th column as at Step \ref {step_1} (Table \ref {rows_rearrange}).
\begin{table}[p]
 \centering
 \begin{tabular}{||c|c||}
  \hhline {|t:==:t|}
  - & - \\ \hhline {||--||}
  $0$ & $0$ \\ \hhline {||--||}
  $u^{w}_{i,j}$ & $v^{w}_{i,j}$ \\
  \hhline {|:==:|}
  - & - \\ \hhline {||--||}
  $0$ & $0$  \\ \hhline {||--||}
  $u^{w}_{k,j}$ & $v^{w}_{k,j}$ \\
  \hhline {|b:==:b|}
 \end{tabular}
 $\Rightarrow$
 \begin{tabular}{||c|c||}
  \hhline {|t:==:t|}
  - & -  \\ \hhline {||--||}
  $z^{w^{*}}_{i,j} = u^{w}_{i,j}$ & $t^{w^{*}}_{i,j} = v^{w}_{i,j}$ \\ \hhline {||--||}
  $0$ & $0$ \\
  \hhline {|:==:|}
  - & - \\ \hhline {||--||}
  $z^{w^{*}}_{k,j} = u^{w}_{k,j}$ & $t^{w^{*}}_{k,j} = v^{w}_{k,j}$  \\ \hhline {||--||}
  $0$ & $0$ \\
  \hhline {|b:==:b|}
 \end{tabular}
 \caption {Rearrange of the rows in the $j$-th column of the block matrix at Step \ref {step_2}}
\label {rows_rearrange}
\end{table}
Again, point $w^{*}$ belongs to $SATP^{2}_{LP}(m,n)$ and has the same value of the objective function.

\item \label {step_3} There exists some $j$ that $x_{i,j}+y_{i,j}=0$. As a result of Steps \ref {step_1} and \ref {step_2} we have
$$\forall i:\ z^{w}_{i,j} + t^{w}_{i,j} > 0, \ z^{w}_{i,j} + u^{w}_{i,j} > 0.$$
Hence, if for some $i$: $z^{w}_{i,j}=0$, then $t^{w}_{i,j} > 0$ and $u^{w}_{i,j} > 0$. We construct the point $w^{*}$ as it is shown in Table \ref {block_ij_epsilon}.
\begin{table}[p]
\centering
	\begin{tabular}{||c|c||}
	\hhline {|t:==:t|}
	 - & - \\ \hhline {||--||}
	 $0$ & $t^{w}_{i,j}$ \\ \hhline {||--||}
	 $u^{w}_{i,j}$ & $v^{w}_{i,j}$ \\
	\hhline {|b:==:b|}
	\end{tabular}
	$\Rightarrow$
	\begin{tabular}{||c|c||}
	\hhline {|t:==:t|}
	 $0$ & $0$ \\ \hhline {||--||}
	 $z^{w^{*}} = \epsilon$ & $t^{w^{*}}_{i,j} = t^{w}_{i,j} - \epsilon$ \\ \hhline {||--||}
	 $u^{w^{*}}_{i,j} = u^{w}_{i,j} - \epsilon$ & $v^{w^{*}}_{i,j} = v^{w}_{i,j} + \epsilon$ \\
	\hhline {|b:==:b|}
	\end{tabular}
	\caption {Construction of the block $i,j$ of the point $w^{*}$}
	\label {block_ij_epsilon}
\end{table}
Since the coordinates $t^{w}_{i,j}$ and $u^{w}_{i,j}$ are nonnegative and $v^{w}_{i,j} < 1$, we can choose a sufficiently small value of $\epsilon$ that $w^{*}$ satisfy the system (\ref {SATeqfirst})-(\ref {SATineq}), (\ref{SATP2_ineq_first}), (\ref{SATP2_ineq_second}). We estimate the value of the objective function
\begin {align*}
f_{c}(w^{*}) &= f_{c}(w) + \epsilon c^{2,1}_{i,j} + \epsilon c^{3,2}_{i,j} - \epsilon  c^{2,2}_{i,j} - \epsilon c^{3,1}_{i,j},\\
f_{c}(w^{*}) &= f_{c}(w) + \epsilon (c^{2,1}_{i,j} + c^{3,2}_{i,j} -  c^{2,2}_{i,j} - c^{3,1}_{i,j}) = f_{c}(w),
\end {align*}
by equation (\ref{obj_vector_IR_for_theorem}). Thus, we can assume that $z^{w}_{i,j} > 0$.

We change the rows in all blocks of the $j$-th column as it's shown in Table \ref {rows_rearrange_in_d_columns}.
\begin{table}[p]
\centering
	\begin{tabular}{||c|c||}
	\hhline {|t:==:t|}
	 $0$ & $0$ \\ \hhline {||--||}
	 $z^{w}_{i,j}$ & $t^{w}_{i,j}$ \\ \hhline {||--||}
	 $u^{w}_{i,j}$ & $v^{w}_{i,j}$ \\
	\hhline {|b:==b:|}
	\end{tabular}
	$\Rightarrow$
	\begin{tabular}{||c|c||}
	\hhline {|t:==:t|}
	 $x^{w^{*}} = z^{w}_{i,j}$ & $y^{w^{*}} = t^{w}_{i,j}$  \\ \hhline {||--||}
	 $z^{w^{*}} = u^{w}_{i,j}$ & $t^{w^{*}} = v^{w}_{i,j}$ \\ \hhline {||--||}
	 $0$ & $0$ \\
	\hhline {|b:==:b|}
	\end{tabular}
	\caption {Rearrange of the rows in the $j$-th column of the block matrix at Step \ref {step_3}}
	\label {rows_rearrange_in_d_columns}
\end {table}

Now in the $j$-th column we have $x^{w^{*}}_{i,j} > 0$ for all $i$. Without loss of generality, we assume that Step \ref {step_3} was applied to the first $d$ columns. Here comes the tricky part: we can't just rearrange rows in such way, as $w^{*}$ may not belong to the polytope $SATP^{2}_{LP}(m,n)$, or the objective function $f_{c}(w^{*})$ has a different value. Therefore, we simply rename the coordinates of the point $w$. Thus, for the first $d$ columns constraints (\ref{obj_vector_IR_for_theorem}) and inequalities (\ref{SATP2_ineq_first}),(\ref{SATP2_ineq_second}) are modified accordingly.

\item \label {step_4} We find an upper-left block $i,j$ with $x^{w}_{i,j} = 0$. As a result of the previous steps, $y^{w}_{i,j}$ is nonnegative, and if $z^{w}_{i,j} = 0$, then both $t^{w}_{i,j}$ and $u^{w}_{i,j}$ are nonnegative. Therefore, we can repeat the $\epsilon$-procedure from Step \ref {step_3} (Table \ref {block_ij_epsilon}) and achieve $z^{w^{*}}_{i,j} > 0$.

\item \label {step_5} The next step depends on the form of the $i$-th row.

\begin {enumerate}
\item If for all $l<j: y^{w}_{i,l}>0$, then after rearrange of the columns in the $i$-th row as at Step \ref {step_1} (Table \ref {columns_rearrange}) we get $x^{w}_{i,l} > 0$ for all $l \leq j$.

\item There exists some $l$ ($d < l < j$) that $y^{w}_{i,l} = 0$. Hence, $x^{w}_{i,l}$ is nonnegative, and if $t^{w}_{i,l} = 0$, then both $z^{w}_{i,j}$ and $v^{w}_{i,j}$ are nonnegative, and we can construct a point $w^{*}$ with $t^{w^{*}}_{i,l} > 0$ by the similar $\epsilon$-procedure. Thus, we assume that $t^{w}_{i,l}$ is nonnegative (Table \ref {fragment_w_matrix}).

\item There exists some $s$ ($s\leq d < j$) that $y^{w}_{i,s} = 0$. Since $s \leq d$, the coordinates in the $s$-th column were renamed at Step \ref {step_3}. Thus, if $t^{w}_{i,s} > 0$, then we can construct a point $w^{*}$ of $SATP^{2}_{LP}(m,n)$ with $y^{w^{*}}_{i,s} > 0$ by the $\epsilon$-procedure. Therefore, we assume that $t^{w}_{i,s} = 0$, and, due to that, $z^{w}_{i,s}$ and $v^{w}_{i,s}$ are nonnegative (Table \ref {fragment_w_matrix}).
\begin{table}[p]
	\centering
	\begin{tabular}{||c|c||c|c||c|c||}
	\hhline {|t:==:t:==:t:==:t|}
	$x_{i,s}$ & $0$ & $x_{i,l}$ & $0$ &  $0$ & $y_{i,j}$ \\ \hhline {||--||--||--||}
	$z_{i,s}$ & $0$ & - & $t_{i,l}$ & $z_{i,j}$ & - \\ \hhline {||--||--||--||}
	- & $v_{i,s}$ & - & - & - & - \\   \hhline {|:==::==::==:|}
	
	$x_{k,s}$ & - & $x_{k,l}$ & - &  $x_{k,j}$ & - \\ \hhline {||--||--||--||}
	- & - & - & - & $0$ & $t_{k,j}$ \\ \hhline {||--||--||--||}
	- & - & - & - & $0$ & - \\ 
  \hhline {|b:==:b:==:b:==:b|}
	\end{tabular}
	\caption {Fragment of the point $w$ block matrix}
	\label {fragment_w_matrix}
\end{table}
\end {enumerate}

\item \label {step_6} Now we examine the $j$-th column.
\begin {enumerate}

\item If for all $k$: $z^{w}_{k,j} > 0$, then we can rearrange the rows in $j$-th column as at Step \ref {step_3} (Table \ref {rows_rearrange_in_d_columns}) and achieve $x^{w^{*}}_{i,j} > 0$ for all $i$. Next, we rename the coordinates for the $j$-th column to become the $(d+1)$-th and increase the value of $d$ by one.

\item There exists some $k$ that $z^{w}_{k,j} = 0$. Then $t^{w}_{k,j}$ is nonnegative, since $z^{w}_{i,j} > 0$. We assume $u^{w}_{k,j} = 0$ (Table \ref {fragment_w_matrix}), otherwise by the $\epsilon$-procedure we can achieve $z^{w^{*}}_{k,j}$ being nonnegative. In this case we can't make $x^{w^{*}}_{i,j}$ nonnegative. 
Let's verify if such point $w$ belongs $SATP^{2}_{LP}(m,n)$ and check the inequality (\ref{SATP2_ineq_first}) for the blocks $i,j,k,l$:
\begin {align*}
(*) &= y_{i,j}+z_{i,j}+u_{i,j}+x_{i,l}+t_{i,l}+v_{i,l}+\\
&+x_{k,j}+t_{k,j}+v_{k,j}+x_{k,l}+t_{k,l}+v_{k,l} \leq 3,\\
&(y_{i,j} = x_{k,j}+y_{k,j},\ \ x_{k,j}+y_{k,j}+t_{k,j}+v_{k,j} = 1),\\
(*) &= 1+z_{i,j}+u_{i,j}+x_{i,l}+t_{i,l}+v_{i,l}+x_{k,j}+x_{k,l}+t_{k,l}+v_{k,l} \leq 3,\\
&(z_{i,j} + u_{i,j} = x_{i,l}+z_{i,l}+u_{i,j},\ \ x_{i,l}+z_{i,l}+t_{i,l}+u_{i,l}+v_{i,l} = 1),\\
(*) &= 2+x_{i,l}+x_{k,j}+x_{k,l}+t_{k,l}+v_{k,l} \leq 3,\\
&(x_{i,l} = x_{k,l}+y_{k,l},\ \ x_{k,j} = x_{k,l}+z_{k,l}+u_{k,l},\\
&x_{k,l}+y_{k,l}+z_{k,l}+t_{k,l}+u_{k,l}+v_{k,l} = 1),\\
(*) &= 3+2x_{k,l} \leq 3.
\end {align*}
By construction, for all $l<j$ we have $x^{w}_{k,l} > 0$, hence, the point $w$ with such blocks $i,j,k,l$ does not belong to the polytope $SATP^{2}_{LP}(m,n)$.

Now we check the inequality (\ref{SATP2_ineq_second}) for the blocks $i,j,k,s$. Note that $s \leq d$, and the $s$-th column was modified at Step \ref{step_3}. Therefore, the inequality has the form 
\begin{align*}
(**) &= y_{i,j}+z_{i,j}+u_{i,j}+x_{i,s}+z_{i,s}+v_{i,s}+\\
&+x_{k,j}+t_{k,j}+v_{k,j}+x_{k,s}+z_{k,s}+v_{k,s} \leq 3, \\ 
&(y_{i,j} = x_{k,j}+y_{k,j},\ \ x_{k,j}+y_{k,j}+t_{k,j}+v_{k,j}=1),\\
(**) &= 1+z_{i,j}+u_{i,j}+x_{i,s}+z_{i,s}+v_{i,s}+x_{k,j}+x_{k,s}+z_{k,s}+v_{k,s}\leq 3,\\
&(v_{i,s}=y_{i,j}+t_{i,j}+v_{i,j},\ \ y_{i,j}+z_{i,j}+t_{i,j}+u_{i,j}+v_{i,j}=1),\\
(**) &= 2+x_{i,s}+z_{i,s}+x_{k,j}+x_{k,s}+z_{k,s}+v_{k,s} \leq 3,\\
&(x_{i,s} = x_{k,s}+y_{k,s}, \ \ z_{i,s}=z_{k,s}+t_{k,s},\ \ x_{k,j}=x_{k,s}+z_{k,s}+u_{k,s},\\
&x_{k,s}+y_{k,s}+z_{k,s}+t_{k,s}+u_{k,s}+v_{k,s} = 1),\\
(**) &= 3+2(x_{k,s} + z_{k,s}) \leq 3.
\end{align*}
Since $x^{w}_{k,s} > 0$, point $w$ with such blocks $i,j,k,s$ does not belong to the polytope $SATP^{2}_{LP}(m,n)$.
\end {enumerate} 

Thereby, the combination of $5$ ($b$ or $c$) and $6$ ($b$) is impossible, and we can repeat the Steps \ref {step_4} - \ref {step_6}, until for all $i,j$ we have $x^{w^{*}}_{i,j} > 0$.
\end {enumerate}

Thus, for any $w$ that maximizes the objective function $f_{c}(x)$ over the polytope $SATP^{2}_{LP} (m,n)$ we can construct such point $w^{*} \in SATP^{2}_{LP} (m,n)$ that $x^{w^{*}}_{i,j} > 0$ for all $i,j$, up to renaming the coordinates, and $f_{c}(w) = f_{c}(w^{*})$.

The point $w^{*}$ can be decomposed into a convex combination
$$w^{*} = \alpha q + (1-\alpha) h,$$
where $0<\alpha \leq 1$, $q$ is an integral vertex of $SATP_{LP} (m,n)$ with $x^{q}_{i,j} = 1$ for all $i,j$, and $h$ has the following coordinates:
\begin {gather*}
x_{i,j}(h)=\frac{x_{i,j}(w^{*})-\alpha }{1-\alpha},\ \  y_{i,j}(h)=\frac{y_{i,j}(w^{*})}{1-\alpha}, \ \ z_{i,j}(h)=\frac{z_{i,j}(w^{*})}{1-\alpha},\\
t_{i,j}(h)=\frac{t_{i,j}(w^{*})}{1-\alpha},\ \ u_{i,j}(h)=\frac{u_{i,j}(w^{*})}{1-\alpha}, \ \ v_{i,j}(h)=\frac{v_{i,j}(w^{*})}{1-\alpha}.
\end {gather*}
The point $h$ satisfies the system (\ref {SATeqfirst})-(\ref {SATineq}), hence, both $q$ and $h$ belongs to $SATP_{LP} (m,n)$.

Our algorithm for integer recognition over $SATP_{LP} (m,n)$ polytope is similar to the one in Lemma \ref {lemma_cut_BQPLP}: if 
$$\max_{x\in SATP_{LP}(m,n)} f_{c}(x) > \max_{x\in SATP^{2}_{LP}(m,n)} f_{c}(x),$$
then, clearly, the maximum is not achieved at an integral vertex, since the polytopes $SATP_{LP} (m,n)$ and $SATP^{2}_{LP} (m,n)$ have the same set of integral vertices, and if
$$\max_{x\in SATP_{LP}(m,n)} f_{c}(x) = \max_{x\in SATP^{2}_{LP}(m,n)} f_{c}(x),$$
then for a point $w$ that maximizes the objective function we can construct such point $w^{*} \in SATP^{2}_{LP} (m,n)$ that
$$f_{c}(w) = f_{c}(w^{*}) = f_{c}(q),$$
where $q$ is an integral vertex, hence, 
$$\max_{x\in SATP_{LP}(m,n)} f_{c}(x) = \max_{z\in SATP(m,n)} f_{c}(z),$$
and the integer recognition problem has a positive answer.

The polytope $SATP^{2}_{LP} (m,n)$ has a polynomial number of additional constraints, therefore, LP over it is polynomially solvable, and the entire algorithm is polynomial. Note that the construction of $w^{*}$ and integral vertex $q$ also requires a polynomial time ($O(mn(m+n))$), as in the worst case for each block $i,j$ we have to check all the blocks in the row $i$ and column $j$.
\end {proof}

%%%%%%%%%%%%%%%%%%%%%%%%%%%%%%%%%%%%%%%%%%

\section {Applications of integer recognition}

In the last section we construct a special polynomially solvable subproblem of some NP-complete problem to show how the constraints and the algorithm from Theorem \ref {integer_recognition_SATP_polynomial} may be used.

We consider a problem of \textit{$2$-$3$ edge constrained bipartite graph coloring} ($2$-$3$-ECBGC): for a given bipartite graph $G = (U,V,E)$ and a function of permitted color combinations for every edge 
$$pc: E \times \{1,2\} \times \{1,2,3\} \rightarrow \{+,-\},$$
it is required to determine if it's possible to assign the vertex colors in such way
$$color: U \rightarrow \{1,2\}\ \mbox {and}\ color: V \rightarrow \{1,2,3\},$$
that they satisfy the constraints of all the edges in the graph.

\begin {Theorem}
$2$-$3$-ECBGC problem is NP-complete.
\label {ECBGC_NPC}
\end {Theorem}

\begin {proof}
The problem obviously belongs to the class NP, as solution can be verified in $O(|E|)$ time. 

We transform exactly-1 3-satisfiability problem to $2$-$3$-ECBGC. Let $m$ be the number of variables and $n$ the number of clauses. First, we construct an instance of integer recognition over $SATP_{LP}(m,n)$ with an objective vector $w \in \R^{6mn}$ as shown in Theorem \ref {reduction_X3SAT}. Then we create a bipartite graph $G_{w}$ with $m$ vertices in $U$ and $n$ vertices in $V$. Graph $G_{w}$ has an edge $(i,j)$ if and only if clause $c_{j}$ has literal $u_{i}$ or $\bar{u}_{i}$. The permitted color combinations are defined as follows:
\begin {gather*}
pc(i,j,k,s) = \left\{
\begin {array} {l}
+,\ \mbox{if} \ w^{k,s}_{i,j} = 1, \\
-,\ \mbox {otherwise}.
\end {array}
\right.
\end {gather*}

There is a bijection between possible color assignments and integral vertices of $SATP(m,n)$:
\begin {gather*}
\forall i \in U: color(i) = \textbf{\emph {row}}_{i} (z) + 1,\\
\forall j \in V: color(j) = \textbf{\emph {col}}_{j} (z) + 1.
\end {gather*}

By Theorem \ref {reduction_X3SAT}, a truth assignment for X3SAT exists if and only if there exists such an integral vertex $z$ of $SATP(m,n)$ that $f_{w}(z) = 3n$. Since some color assignment satisfies the permitted color constraints of the edge $i,j$ if and only if $f_{w}(z_{i,j}) = 1$, and there are exactly $3n$ edges in the graph $G_{w}$, we have $$\mbox{X3SAT} \leq_{p} \mbox{2-3-ECBGC}.$$
\end {proof}

Using Theorem \ref {integer_recognition_SATP_polynomial}, we construct a special polynomially solvable subproblem of $2$-$3$-edge constrained bipartite graph coloring.

\begin {Theorem}
$2$-$3$-ECBGC problem is polynomially solvable if the function of permitted color combinations satisfies the following constraints
\begin {gather}
\nonumber
\forall j\in V, \ \exists a_{j},b_{j} \in \{1,2,3\}\ (a \neq b), \forall i\in U:\\ 
pc (i,j,a_{j},1) = pc (i,j,b_{j},2) = ``+" \ \Leftrightarrow \ pc (i,j,a_{j},2) = pc (i,j,b_{j},1) = ``+".
\label {permitted_colors}
\end {gather}
\label{ECBGC_polynomial}
\end {Theorem}

\begin {proof}
Let $|U| = m$ and $|V| = n$. We reduce $2$-$3$-ECBGC problem to integer recognition over $SATP_{LP}(m,n)$ by constructing an objective vector $c \in \R^{6mn}$ from the function of permitted color combinations as follows:
\begin {gather*}
c^{k,s}_{i,j} = \left\{
\begin {array} {l}
1,\ \mbox{if} \ pc(i,j,k,s) = ``+", \\
-1 ,\ \mbox {in the case of \textit{zero balancing}},\\
0 ,\ \mbox {otherwise}.
\end {array}
\right.
\end {gather*}

We have a \textit{zero balancing} case if an edge $i,j$ out of four color combinations $(a_{j},1)$, $(a_{j},2)$, $(b_{j},1)$, and $(b_{j},2)$ has only one permitted. Assume, without loss of generality, that it is $(a_{j},1)$, then we assign $c^{b_{j},2}_{i,j} = -1$ to achieve zero balance:
$$c^{a_{j},1}_{i,j} + c^{b_{j},2}_{i,j} = c^{a_{j},2}_{i,j} + c^{b_{j},1}_{i,j} = 0.$$
An example of an objective vector $c$ for $a_{1} = a_{2} = 1$ and $b_{1} = b_{2} = 2$ is shown in Table \ref {colors_to_vector_example}.
\begin{table}[h]
\centering
\begin{tabular}{||c|c||c|c||}
\hhline {|t:==:t:==:t|}
$+$ & $-$ & $+$ & $-$ \\ \hhline {||--||--||}
$+$ & $-$ & $-$ & $-$ \\ \hhline {||--||--||}
$-$ & $+$ & $+$ & $+$ \\ 
\hhline {|:==::==:|}
$-$ & $-$ & $+$ & $+$ \\ \hhline {||--||--||}
$-$ & $+$ & $-$ & $-$ \\ \hhline {||--||--||}
$+$ & $-$ & $-$ & $+$ \\ 
\hhline {|b:==:b:==:b|}
\end{tabular}
$\Rightarrow$
\begin{tabular}{||c|c||c|c||}
\hhline {|t:==:t:==:t|}
$1$ & $\ 0 \ $ & $\ 1\ $ & $0$ \\ \hhline {||--||--||}
$1$ & $0$ & $0$ & $-1$ \\ \hhline {||--||--||}
$0$ & $1$ & $1$ & $1$ \\ 
\hhline {|:==::==:|}
$-1$ & $0$ & $1$ & $1$ \\ \hhline {||--||--||}
$0$ & $1$ & $0$ & $0$ \\ \hhline {||--||--||}
$1$ & $0$ & $0$ & $1$ \\ 
\hhline {|b:==:b:==:b|}
\end{tabular}
\caption {Example of an objective vector $c$}
\label {colors_to_vector_example}
\end{table}

For every edge $i,j$ there are $4$ possible color combinations that include the colors $a_{j}$ and $b_{j}$. There are $16$ possible constraints on these color combinations. Six of them are forbidden by (\ref{permitted_colors}). Others transform into a block of vector $c$ of the form (\ref{obj_vector_IR}). We again use the bijection between integral vertices and possible color assignments as in Theorem \ref {ECBGC_NPC}. 

Thus, a permitted color assignment for $2$-$3$-ECBGC exists if and only if
$$\max_{x\in SATP_{LP}(m,n)} f_{c}(x) = \max_{z\in SATP(m,n)} f_{c}(z) = |E|.$$
Integer recognition over $SATP_{LP}(m,n)$ with the objective function $f_{c}(x)$ is polynomially solvable (Theorem \ref {integer_recognition_SATP_polynomial}), therefore, such instance of $2$-$3$-ECBGC problem is polynomially solvable as well.
\end {proof}

Note that the constrained objective function (\ref{obj_vector_IR}) is far more flexible than we used for $2$-$3$-ECBGC problem, since it is not limited to the $\{-1,0,1\}$ values. For example, we can add the weight for permitted color combinations that will satisfy the constraints (\ref{obj_vector_IR}) and solve the problem by integer recognition algorithm.

%%%%%%%%%%%%%%%%%%%%%%%%%%%%%%%%%%%%%%%%%%

\section{Conclusions}

We have considered $SATP(m,n)$ polytope and its LP relaxation $SATP_{LP}(m,n)$. This polytope is a simple extension of the well-known and important Boolean quadric polytope $BQP(n)$, constructed by adding two additional coordinates per block. Polytope $SATP(m,n)$ is the object of our interest, since several special instances of 3-SAT like X3SAT and NAE-3SAT are reduced to integer programming and integer recognition over it.

We have compared key properties of the Boolean quadric polytope and 3-SAT polytope. Like the $BQP(n)$, polytope $SATP(m,n)$ has an exponential extension complexity, LP relaxation $SATP_{LP}(m,n)$ is quasi-integral with respect to $SATP(m,n)$, 1-skeleton of $SATP(m,n)$ is not a complete graph, but is a very dense one, with the diameter equals $2$, and the clique number being superpolynomial in dimension. Unlike the $BQP(n)$, denominators of the fractional vertices of $SATP_{LP}(m,n)$ relaxation can take any positive integral values, and integer recognition over $SATP_{LP}(m,n)$ is NP-complete.

Finally, we have considered possible constraints on the objective function for which integer recognition over $SATP_{LP}(m,n)$ is polynomially solvable. We have introduced a problem of $2$-$3$ edge constrained bipartite graph coloring that is NP-complete in general case, and design a polynomial time algorithm for its special subproblem, based on $SATP_{LP}(m,n)$ properties. This example shows how the polytope $SATP(m,n)$ may be used, and why it is of interest for further studying.

%%%%%%%%%%%%%%%%%%%%%%%%%%%%%%%%%%%%%%%%%%

\bibliographystyle{amsplain}

\end{document}